\documentclass[12pt]{amsart}
\usepackage{enumerate,amssymb,version,aliascnt,geometry,stmaryrd,mathrsfs,chngcntr}
\usepackage[all]{xy}
\usepackage[mathscr]{euscript}
\usepackage{hyperref}
\hypersetup{
pdfauthor={Olivier Haution},
pdftitle={Actions of diagonalizable p-groups and Chern numbers modulo p},
pdfkeywords={},
hidelinks,
unicode
}

\usepackage{textcomp}

\swapnumbers

\DeclareMathOperator{\Spec}{Spec}
\DeclareMathOperator{\im}{im}
\DeclareMathOperator{\CH}{CH}
\DeclareMathOperator{\Ch}{Ch}

\newcommand{\character}{\mathcal{L}}
\newcommand{\Fil}{\mathcal{C}}
\newcommand{\Eu}{\mathcal{E}}
\newcommand{\length}{\ell}
\newcommand{\Speck}{k}

\newcommand{\Lc}{\mathcal{L}}
\newcommand{\bb}{\mathbf{b}}
\newcommand{\lc}{\llbracket}
\newcommand{\rc}{\rrbracket}
\newcommand{\mup}{{\mu_p}}
\newcommand{\Zz}{\mathbb{Z}}
\newcommand{\Pp}{\mathbb{P}}
\newcommand{\Laz}{\mathbb{L}}
\newcommand{\Nn}{\mathbb{N}}

\newcommand{\Fd}{\mathbb{F}_2}
\newcommand{\Fp}{\mathbb{F}_p}
\newcommand{\Tan}{T}

\newtheorem{theorem}{Theorem}[section]
\newaliascnt{proposition}{theorem}
\newtheorem{proposition}[proposition]{Proposition}
\newaliascnt{lemma}{theorem}
\newtheorem{lemma}[lemma]{Lemma}
\newaliascnt{corollary}{theorem}
\newtheorem{corollary}[corollary]{Corollary}
\newtheorem*{theorem*}{Theorem}
\newtheorem*{proposition*}{Proposition}

\theoremstyle{definition}
\newaliascnt{remark}{theorem}
\newtheorem{remark}[theorem]{Remark}
\newaliascnt{example}{theorem}
\newtheorem{example}[example]{Example}
\newaliascnt{definition}{theorem}
\newtheorem{definition}[definition]{Definition}

\newtheoremstyle{par}
  {}
  {}
  {}
  {}
  {}
  {.}
  { }
  {}%
\theoremstyle{par}
\newtheorem{para}[theorem]{}

\counterwithin{equation}{theorem}
\numberwithin{equation}{theorem}

\newcommand{\rref}[1]{(\ref{#1})}
\newcommand{\dref}[2]{(\ref{#1}.\ref{#2})}

\begin{document}
\begin{abstract}
We obtain lower bounds for the dimension of fixed loci of diagonalizable \(p\)-groups acting on smooth projective varieties.
Those bounds depend on the modulo \(p\) Chern numbers of the ambient variety, and are expressed in a natural way by introducing an appropriate filtration on the ``modulo \(p\) cobordism ring'' (for \(p=2\) this is Thom's unoriented cobordism ring \(MO^*\)).
They are obtained using equivariant localization methods, via the concentration theorem for the Chow ring, and by a technique of ``partition dividing''.
As applications we derive statements in the spirit of Boardman's Five-Halves Theorem for involutions on manifolds.
\end{abstract}

\author{Olivier Haution}
\title{Actions of diagonalizable $p$-groups and Chern numbers modulo $p$}
\address{Dipartimento di Matematica e Applicazioni, Università degli Studi di Milano-Bicocca, via Roberto Cozzi 55, 20125 Milano, Italy}
\email{olivier.haution@unimib.it}

\subjclass[2010]{}

\date{\today}

\maketitle

\numberwithin{theorem}{section}
\numberwithin{lemma}{section}
\numberwithin{proposition}{section}
\numberwithin{corollary}{section}
\numberwithin{example}{section}
\numberwithin{definition}{section}
\numberwithin{remark}{section}

\section*{Introduction}

Consider a smooth projective \(k\)-variety \(X\) with an action of a diagonalizable group \(G\).
We will assume that the cardinality of the character group \(\widehat{G}\) is a power \(q\) of a prime number \(p\).
Explicitly \(G \simeq \mu_{p^{r_1}} \times \cdots \times \mu_{p^{r_m}}\) for some integers \(r_1,\dots,r_m\), and \(q = p^{r_1} \cdots p^{r_m}\).
When \(k\) is algebraically closed of characteristic not \(p\), the group \(G\) is a constant finite abelian \(p\)-group of cardinality \(q\), which acts on \(X\) by \(k\)-automorphisms.
When  \(k\) has characteristic \(p\), the group \(G\) is infinitesimal, and so the notion of \(G\)-action has a somewhat different flavor and can be interpreted in terms of global derivations (see e.g.\ \cite[(4.5.ii)]{inv} for the case \(G=\mu_2\)).

The fixed point theorem of \cite{fpt} asserts that if a Chern number of \(X\) is prime to \(p\), then the group \(G\) must fix a point on \(X\).
In this paper we refine that statement, building on the fact that the Chern numbers are indexed by partitions.
We obtain a lower bound on the dimension of the fixed locus \(X^G\), which will depend on the exact shape of the partition whose associated Chern number is prime to \(p\).
This bound is zero in the worst case --- for ``long partitions'' --- thus simply recovering the fixed point theorem.
But in general we obtain more than the mere existence of a fixed point.

Our dimensional bounds are reminiscent of Boardman's Five-Halves theorem for involutions in topology \cite[Theorem 1]{Boardman-BAMS}, although it should be noted that no assumption on the shape of the partition was present in Boardman's theorem.
The case of \(\Zz/2\)-actions in algebraic geometry was investigated in \cite{ciequ} (over fields of characteristic different from \(2\)), where in particular an algebraic version of Boardman's theorem was obtained.
This paper arose as an attempt to better understand the situation for actions of abelian \(p\)-groups \(G\neq \Zz/2\), in the algebraic context.\\

In order to state our main result, it will be convenient to introduce an appropriate cobordism ring.
To do so, let us identify two smooth projective \(k\)-varieties if they have the same Chern numbers modulo \(p\), indexed by the partitions.
In this way we obtain an \(\Fp\)-algebra \(\Laz_p\), the ``modulo \(p\) cobordism ring'', in which each smooth projective \(k\)-variety \(X\) has a class \(\lc X \rc\).
This algebra is graded, whereby \(\lc X\rc\) is homogeneous of degree \(-n\) when \(X\) has pure dimension \(n\).
It is known that, independently of the field \(k\), the \(\Fp\)-algebra \(\Laz_p\) is (non-canonically) polynomial on variables \(\ell_i \in \Laz_p^{-i}\), where \(i\) runs over the set
\[
N_p=\{i \in \Nn, \text{ where \(i+1\) is not a power of \(p\)}\}.
\]
Let us define another grading of the ring \(\Laz_p\) by letting each \(\ell_i\) be homogeneous of degree \(\lfloor i/q \rfloor\), and denote by \(\dim_q x\) the degree of an element \(x \in \Laz_p\) with respect to this grading.
Even though the grading depends on the choice of the variables \(\ell_i\), the function \(\dim_q\) does not.
Recalling that \(q\) is the number of characters of \(G\), and that \(X^G\) denotes the fixed locus for the \(G\)-action on \(X\), the main result can now be stated as follows:

\begin{theorem*}
If \(G\) acts on a smooth projective \(k\)-variety \(X\), we have
\[
\dim X^G \geq \dim_q \lc X\rc.
\]
\end{theorem*}
Thus, while the usual grading of \(\Laz_p\) encodes information on the dimension of the ambient variety, the new grading encodes information on the dimension of the fixed locus.
Note however that the class of the variety \(X\) in \(\Laz_p\) does not depend on the \(G\)-action, thus the  theorem provides restrictions on the fixed locus dimension for all possible \(G\)-actions, and those restrictions come from the nonequivariant geometry of \(X\).
Heuristically, the general pattern is that varieties whose cobordism class is highly nontrivial cannot carry actions with fixed locus of low dimension.

Various concrete applications of the theorem are discussed in \S\ref{sect:app}.
We attempt to control (from below) the ratio \(n/d\), where \(n\) is the dimension of the ambient variety \(X\), and \(d\) the dimension of the fixed locus \(X^G\).
Our lower bound for this ratio is approximately \(1/q\), reaching this value in the most favorable cases, but in general depending on the cobordism class of the ambient variety.
We also find explicit divisibility conditions on the cobordism classes of varieties having very low dimensional fixed locus.

Looking at the special case of \(\mu_2\)-actions, the algebraic version of Boardman's theorem arises from the exceptional fact that \(\dim_2\lc X \rc \geq (2/5) \dim X\) as soon as \(\lc X \rc \neq 0\); this type of unconditional bound is unique to the case \(q=2\).\\

Unsurprisingly, to prove the theorem we use equivariant localization methods, via the concentration theorem from \cite{conc}.
The other main idea is to ``deform'' equivariantly the Conner--Floyd Chern classes (those are the cycle classes whose degrees are the Chern numbers), in such a way that their restrictions to the fixed locus have to be supported in large enough codimension.
We achieve this by ``dividing'' the partitions, spreading them out across all \(q\) characters of \(G\) in order to account for all possible actions on the normal bundle to the fixed locus; this is the origin of the ratio \(1/q\) mentioned above.

A key tool is a certain morphism \(\epsilon_1\) from the \(\mup\)-equivariant Chow group (modulo \(p\)) of a variety with trivial \(\mup\)-action to its nonequivariant Chow group.
It clearly plays a central role, but at the same time remains somewhat mysterious, as we lack a clear geometric interpretation of it (for instance it is \emph{not} a graded ring morphism).\\

\textbf{Relation to \cite{diag_cobord}.}
After the present paper was written, its results were substantially extended in \cite{diag_cobord}, where we determine all the restrictions on the dimension of the fixed locus arising from the Chern numbers of the ambient variety, for actions of diagonalizable \(p\)-groups over fields of characteristic different from \(p\).
The results of \cite{diag_cobord} are expressed in terms of the class of the variety in the Lazard ring, rather than in the ring \(\Laz_p\) considered in the present paper, and thus capture in addition the information carried by Chern numbers modulo higher powers of \(p\).
In particular, when the characteristic of \(k\) differs from \(p\), the bounds obtained here (see (\ref{cor:floor_q}), (\ref{th:deg_q})) are contained in \cite[Corollary~3, Theorem~2]{diag_cobord}.
For the complete determination of the restrictions arising from Chern numbers, we therefore refer the reader to \cite{diag_cobord}.

The present paper is nevertheless retained, as its methods are genuinely different (it uses only the modulo \(p\) equivariant Chow ring, whereas \cite{diag_cobord} rests on an analysis of the equivariant algebraic cobordism ring) and as it covers base fields of characteristic \(p\), where the diagonalizable group \(G\) is infinitesimal, a situation currently out of reach of the methods of \cite{diag_cobord}.

\section{Generalized Conner--Floyd Chern classes}
Let \(k\) be a field.
We will often also denote by \(k\) its spectrum.
By a \(k\)-variety we intend a quasi-projective scheme of finite type over \(k\).\\

In this section we fix a linear algebraic group \(G\) over \(k\), and a commutative ring \(R\).

\begin{para}
For a smooth \(k\)-variety \(X\) with a \(G\)-action, we consider its equivariant Chow ring \(\CH_G(X)\) (see \cite{EG-Equ}), and set
\[
A(X)= \CH_G(X) \otimes_{\Zz} R, \quad \text{ and } \quad A=A(\Speck).
\]
Using the codimension grading of the equivariant Chow ring, we may view \(A(X)\) as a graded ring, whose degree \(n\) component we denote by \(A^n(X)\).
\end{para}

\begin{para}
The equivariant Chow ring verifies the equivariant version of the projective bundle theorem, and admits Chern classes for equivariant vector bundles satisfying the Whitney sum formula.
This implies that one may use the splitting principle for the equivariant Chow ring.
\end{para}

\begin{para}
A \emph{partition} \(\alpha\) is a family \((\alpha_1,\dots,\alpha_m)\), where \(\alpha_i\) are integers satisfying \(\alpha_1 \geq \alpha_2 \geq \cdots \geq \alpha_m \geq 1\).
We include the case of the empty partition \(\alpha= \varnothing\), when \(m=0\).
The integer \(m\) is the \emph{length} of \(\alpha\), denoted by \(\length(\alpha)\).
The \emph{weight} of \(\alpha\) is the integer
\[
|\alpha| = \alpha_1 + \cdots + \alpha_m.
\]
\end{para}

\begin{para}
If \(\alpha=(\alpha_1,\dots,\alpha_m)\) and \(\beta=(\beta_1,\dots,\beta_s)\) are partitions, we denote by \(\alpha \cup \beta\) the partition obtained by reordering the \((m+s)\)-tuple \((\alpha_1,\dots,\alpha_m,\beta_1,\dots,\beta_s)\).
\end{para}

\begin{para}
\label{p:bb}
When \(B\) is a commutative ring, we denote by \(B[[\bb]]\) the ring of power series with coefficients in \(B\) in the variables \(b_i\) for \(i\in \Nn \smallsetminus \{0\}\).
It will be convenient to set \(b_0=1\).
To each partition \(\alpha=(\alpha_1,\dots,\alpha_m)\) corresponds the monomial
\[
b_{\alpha} = b_{\alpha_1} \cdots b_{\alpha_m}.
\]
\end{para}

\begin{para}
When \(X\) is a \(k\)-variety, we denote by \(K_0(X)\) its Grothendieck group of vector bundles, and when \(G\) acts on \(X\), by \(K_0(X;G)\) its Grothendieck group of \(G\)-equivariant vector bundles.
\end{para}

\begin{definition}
\label{def:gen_CF}
Let \(g=(g_i, i\in \Nn)\) be a collection of polynomials in \(A[x]\) such that \(g_0=1\).
In view of the splitting principle, there exists a unique way to construct for each smooth \(k\)-variety \(X\) with a \(G\)-action a map
\[
P_g \colon K_0(X;G) \to A(X)[[\bb]],
\]
satisfying the following conditions:
\begin{enumerate}[(1)]
\item If \(L \to X\) is a \(G\)-equivariant line bundle, we have
\[
P_g(L) = \sum_{i\in\Nn} g_i(c_1(L))b_i \quad \text{(recall that \(b_0=1\))}.
\]

\item If \(f \colon Y \to X\) is a \(G\)-equivariant morphism between smooth \(k\)-varieties, and \(E \in  K_0(X;G)\), then
\[
P_g(f^*E) = f^*(P_g(E)).
\]

\item For any \(E,F \in K_0(X;G)\), we have
\[
P_g(E + F) = P_g(E) \cdot P_g(F).
\]
\end{enumerate}
For a partition \(\alpha\), we let \(g_\alpha(E) \in A(X)\) be the \(b_\alpha\)-coefficient of \(P_g(E)\).
\end{definition}

\begin{example}
\label{ex:CF}
If \(g_i=x^i\) for each \(i \in \Nn\), then \(g_\alpha\) is the Conner--Floyd Chern class \(c_\alpha\) for each partition \(\alpha\).
\end{example}

\begin{para}
\label{p:g_graded}
Let us view the polynomial ring \(A[x]\) as a \(\Zz\)-graded ring using the natural grading on \(A\), and letting \(x\) be homogeneous of degree \(1\).
Recall that \(c_1(L) \in A^1(X)\) when \(L \to X\) is a \(G\)-equivariant line bundle.
It follows from the splitting principle that, if each \(g_i(x)\) is homogeneous of degree \(i\), then for each partition \(\alpha\) and \(E \in K_0(X;G)\) we have \(g_{\alpha}(E) \in A^{|\alpha|}(X)\).
\end{para}

\begin{para}
\label{p:g_alpha_sum}
Observe that, for any partition \(\alpha\) and \(E,F \in K_0(X;G)\)
\[
g_\alpha(E + F) = \sum_{\beta \cup \gamma = \alpha} g_\beta(E) g_\gamma(F).
\]
\end{para}

\begin{para}
\label{p:trivial_summand}
Note that \(P_g(1)=1\), and so from \rref{p:g_alpha_sum} we deduce that, for any partition \(\alpha\) and \(E \in K_0(X;G)\)
\[
g_\alpha(E \oplus 1) = g_\alpha(E).
\]
\end{para}

\begin{para}
\label{p:g_line}
Assume that \(E = L_1 \oplus \dots \oplus L_r\), where \(L_j\) are \(G\)-equivariant line bundles.
Let \(\alpha=(\alpha_1,\dots,\alpha_m)\) be a partition.
Then
\begin{equation}
\label{eq:g_line}
g_\alpha(E) = \sum_{i_1,\dots,i_m} g_{\alpha_1}(c_1(L_{i_1})) \cdots g_{\alpha_m}(c_1(L_{i_m})),
\end{equation}
where \(i_1,\dots,i_m\) run over the \(m\)-tuples of pairwise distinct elements in \(\{1,\dots,r\}\).
\end{para}

\begin{para}
\label{p:g:length_rank}
Let \(E \to X\) be a \(G\)-equivariant vector bundle of rank \(r \in \Nn\), where \(X\) is smooth.
Then it follows from the splitting principle and the formula \eqref{eq:g_line} that \(g_\alpha(E)=0\) whenever \(\length(\alpha) >r\).
\end{para}

\begin{para}
\label{p:g_inverse}
If \(E \in K_0(X;G)\), and \(\alpha \neq \varnothing\), we have \(g_\alpha(E-E)=0\) and so by \rref{p:g_alpha_sum} 
\[
0=\sum_{\beta \cup \gamma = \alpha} g_\beta(E) g_\gamma(-E).
\]
\end{para}

\begin{para}
\label{p:g_(n)}
Let \(n \in \Nn \smallsetminus \{0\}\).
Then it follows from \rref{p:g_alpha_sum} that, for any \(E,F \in K_0(X;G)\) we have
\[
g_{(n)}(E +F)=g_{(n)}(E) + g_{(n)}(F).
\]
Moreover if
\[
g_n(x) = \sum_{i=1}^r \lambda_i x^i \text{ with \(\lambda_1,\dots,\lambda_r \in A\)},
\]
then, for any \(E \in K_0(X;G)\) we have 
\[
g_{(n)}(E) = \sum_{i=1}^r \lambda_i c_{(i)}(E).
\]
\end{para}

\begin{para}
\label{p:ineq_part}
Let \(\alpha=(\alpha_1,\dots,\alpha_m)\) and \(\beta=(\beta_1,\dots,\beta_s)\) be partitions.
We will write \(\alpha \geq \beta\) when \(m\geq s\) and \(\alpha_j \geq \beta_j\) for all \(j=1,\dots,s\).
\end{para}

\begin{lemma}
\label{lemm:g:div}
Assume that for each \(i\in \Nn\), the polynomial \(g_i\) is divisible by \(x^{u_i}\), where \(u_i \in \Nn\).
Let \(E \to X\) be a \(G\)-equivariant vector bundle, where \(X\) is a smooth \(k\)-variety, and let \(\alpha = (\alpha_1,\dots,\alpha_m)\) be a partition.

Then \(g_\alpha(E)\) is an \(A\)-linear combination of classes \(c_\beta(E)\) for \(\beta \geq \tilde{u}\) (in the sense of \rref{p:ineq_part}), where \(\tilde{u}\) is the partition corresponding (upon reordering and removing zeroes) to the \(m\)-tuple \((u_{\alpha_1},\dots,u_{\alpha_m})\).
\end{lemma}
\begin{proof}
The \(k\)-variety \(X\) decomposes \(G\)-equivariantly as \(X=X_0 \sqcup \dots \sqcup X_r\), where \(E|_{X_i}\) has constant rank \(i\) for each \(i=0,\dots,r\).
Letting \(E'\to X\) be the \(G\)-equivariant vector bundle such that \(E'|_{X_i} = E|_{X_i} \oplus 1^{\oplus r-i}\) for each \(i\), in view of \rref{p:trivial_summand} we may replace \(E\) with \(E'\), and thus assume that \(E\) has constant rank \(r\).
We consider the action of the symmetric group \(S_r\) on the polynomial ring \(A[x_1,\dots,x_r]\) by permuting the variables \(x_1,\dots,x_r\).
For a partition \(\beta = (\beta_1,\dots,\beta_s)\) with \(s \leq r\), let us denote by \(o(\beta)\) the \(S_r\)-orbit of the monomial \(x_1^{\beta_1}\cdots x_s^{\beta_s}\) and consider the symmetric polynomial
\[
R_\beta = \sum_{M \in o(\beta)} M \in A[x_1,\dots,x_r].
\]
The polynomials \(R_\beta\), where \(\beta\) runs over the partitions such that \(\ell(\beta)\leq r\), form an \(A\)-basis of the ring of symmetric polynomials \(A[x_1,\dots,x_r]^{S_r}\); with respect to that basis the \(R_\beta\)-coefficient of a symmetric polynomial \(P\) is zero unless \(P\) contains as a monomial a nonzero multiple of an element of \(o(\beta)\).

By the splitting principle, we may assume that \(E=L_1\oplus \dots \oplus L_r\), where each \(L_j\) is a \(G\)-equivariant line bundle.
Let us consider the formula \eqref{eq:g_line}, and let \(i_1,\dots,i_m\) be pairwise distinct elements of \(\{1,\dots,r\}\).
In particular we have \(m \leq r\).
The polynomial 
\[
g_{\alpha_1}(x_{i_1}) \cdots g_{\alpha_m}(x_{i_m}) \in A[x_1,\dots,x_r]
\]
is an \(A\)-linear combination of monomials of the form 
\[
x_{i_1}^{v_1}\cdots x_{i_m}^{v_m}, \quad \text{ with \(v_j \geq u_{\alpha_j}\) for all \(j=1,\dots,m\)}.
\]
Such a monomial belongs to the orbit \(o(\beta)\) for a unique partition \(\beta\) such that \(\ell(\beta)\leq r\), namely the partition \(\beta\) corresponding to the tuple \((v_1,\dots,v_m\)) (upon reordering and removing zeroes).
Note that we then have \(\beta \geq \tilde{u}\).
Thus the symmetric polynomial 
\[
\sum_{i_1,\dots,i_m} g_{\alpha_1}(x_{i_1}) \cdots g_{\alpha_m}(x_{i_m}),
\]
where \(i_1,\dots,i_m\) run over the pairwise distinct elements of \(\{1,\dots,r\}\), is an \(A\)-linear combination of the symmetric polynomials \(R_{\beta}\), for \(\beta \geq \tilde{u}\) and \(\ell(\beta) \leq r\).
From the formula \eqref{eq:g_line} we deduce that \(g_\alpha(E)\) is an \(A\)-linear combination of the elements
\begin{equation}
\label{eq:r_beta}
R_\beta(c_1(L_1),\dots,c_r(L_r)), \quad \text{for \(\beta \geq \tilde{u}\) and \(\ell(\beta)\leq r\)}.
\end{equation}
To conclude, observe that \eqref{eq:r_beta} coincides with the Conner--Floyd Chern class \(c_\beta(E)\).
\end{proof}

\section{\texorpdfstring{\(\mu_p\)-actions}{{\textmu}p-actions}}
\begin{definition}
We now fix a prime number \(p\).
For a smooth \(k\)-variety \(X\), we will denote by \(\Ch(X)=\CH(X)/p\) its modulo \(p\) Chow ring.
If a linear algebraic group \(G\) over \(k\) acts on \(X\), we denote by \(\Ch_G(X)=\CH_G(X)\otimes_{\Zz} \Zz/p\) the \(G\)-equivariant modulo \(p\) Chow ring.
For \(i \in \Zz\), the degree \(i\) components will be denoted by \(\Ch^i(X)\) and \(\Ch^i_G(X)\).
\end{definition}

\begin{para}
\label{p:epsilon}
When a linear algebraic group \(G\) acts on a smooth \(k\)-variety \(X\), there is a forgetful morphism of graded rings \(\epsilon \colon \Ch_G(X) \to \Ch(X)\) which commutes with pushforwards and pullbacks.
\end{para}

\begin{para}
\label{p:Ch_trivial_alg}
When \(G\) acts trivially on a smooth \(k\)-variety \(X\), the morphism of graded rings \(\epsilon\) admits a natural section
\[
\Ch(X) \to \Ch_G(X),
\]
and so we will view \(\Ch_G(X)\) as a graded \(\Ch(X)\)-algebra.
\end{para}

\begin{para}
\label{p:Lc}
When \(G\) is a linear algebraic group over \(k\), we denote \(\widehat{G}\) its character group (viewed as an abstract group).
For \(c \in \widehat{G}\), we denote by \(\character_c\) the \(G\)-equivariant line bundle over \(\Spec k\), given by \(\mathbb{A}^1\) with the linear \(G\)-action via \(c\).
Its pullback to any \(k\)-variety with a \(G\)-action will again be denoted by \(\character_c\).
\end{para}

\begin{para}
\label{p:diag_grading}
Let \(C\) be a diagonalizable group of finite type over \(k\), and let \(X\) be a \(k\)-variety with trivial \(C\)-action.
If \(E \to X\) is a \(C\)-equivariant vector bundle, we have a \(C\)-equivariant decomposition
\[
E=\bigoplus_{c \in \widehat{C}} E(c),  
\]
where \(C\) acts via the character \(c\) on the subbundle \(E(c)\).
\end{para}

\begin{para}
\label{p:fixed_diag_0}
In the situation of \rref{p:diag_grading}, we will write \(E^C=E(0)\).
\end{para}

\begin{lemma}
\label{lemm:decomp_diag}
Let \(G\) be a diagonalizable group of finite type over \(k\) acting on a \(k\)-variety \(X\).
Let \(C \subset G\) be a subgroup acting trivially on \(X\).
Let \(E \to X\) be a \(G\)-equivariant vector bundle.
Then there exist \(G/C\)-equivariant vector bundles \(F_c\) for \(c \in \widehat{C}\) such that
\[
E \simeq \bigoplus_{c \in \widehat{C}} F_c \otimes \character_c
\]
as \(C\)-equivariant vector bundles.
\end{lemma}
\begin{proof}
By \cite[(2.4.5)]{fpt} the decomposition of \rref{p:diag_grading} is \(G\)-equivariant.
For each \(c \in \widehat{C}\), pick a preimage \(g_c \in \widehat{G}\) of \(c\) under the surjection \(\widehat{G} \to \widehat{C}\).
Then \(C\) acts trivially on \(F_c=E(c) \otimes (\character_{g_c})^\vee\), for each \(c\in \widehat{C}\), and we obtain the required decomposition.
\end{proof}

\begin{para}(See e.g.\ \cite[Theorem~2.10]{Totaro-group_cohomology}.)
\label{p:ch_mup}
When \(\mup\) acts trivially on a smooth \(k\)-variety \(X\), there is an isomorphism of \(\Ch(X)\)-algebras (see \rref{p:Ch_trivial_alg})
\[
\Ch(X)[t] \to \Ch_{\mup}(X), \quad t \mapsto c_1(\character_1).
\]
\end{para}

\begin{definition}
\label{def:epsilon_r}
Assume that \(\mu_p\) acts trivially on a smooth \(k\)-variety \(X\), and let \(r\in \Fp\).
In view of \rref{p:ch_mup} we may define a morphism of \(\Ch(X)\)-algebras
\[
\epsilon_r \colon \Ch_{\mu_p}(X) \to \Ch(X)
\]
by the condition \(c_1(\character_1) \mapsto r\).
Note that this morphism is graded only when \(r = 0\).
When \(r\neq 0\), this extends to a morphism
\[
\epsilon_r \colon \Ch_{\mu_p}(X)[\Eu_{\mu_p}^{-1}] \to \Ch(X),
\]
where \(\Eu_{\mu_p}\) is the multiplicative system generated by the classes \(c_1(\character_c)\) for \(c \in (\Zz/p) \smallsetminus \{0\}\).
(In fact it suffices to invert \(c_1(\character_c)\) for a single \(c\neq 0\).)
\end{definition}

\begin{para}
\label{p:epsilon_retraction}
For each \(r \in \Fp\) the morphism \(\epsilon_r\) is a retraction of the natural morphism \(\Ch(X) \to \Ch_{\mup}(X)\) (see \rref{p:Ch_trivial_alg}).
\end{para}

\begin{para}
\label{p:epsilon_indep}
The morphism \(\Ch(X) \to \Ch_{\mup}(X)\) of \rref{p:Ch_trivial_alg} is graded, and an isomorphism in degrees \(\leq 0\).
It thus follows from \rref{p:epsilon_retraction} that the restriction of \(\epsilon_r\) to \(\Ch_G^i(X)\) does not depend on \(r \in \Fp\) when \(i \leq 0\).
\end{para}

\begin{para}
\label{p:epsilon_epsilon_0}
The morphism \(\epsilon_0\) coincides with the morphism \(\epsilon\) of \rref{p:epsilon}, when \(\mup\) acts trivially on \(X\).
\end{para}

\begin{definition}
\label{def:Fil_G}
When a linear algebraic group \(G\) over \(k\) acts on a smooth \(k\)-variety \(X\), we consider the subring
\[
\Fil_G(X) \subset \Ch(X),
\]
generated by the Chern classes \(c_i(E)\), where \(E \in \im(K_0(X;G) \to K_0(X))\) and \(i \in \Nn\).
\end{definition}

\begin{para}
Let \(X\) be a smooth \(k\)-variety with an action of a linear algebraic group \(G\) over \(k\).
Recall that the \emph{Euler class} of a \(G\)-equivariant vector bundle \(E \to X\) is defined as
\[
e(E) = s^* \circ s_*(1) \in \Ch_G(X),
\]
where \(s\colon X \to E\) is the zero-section.
When \(E\) has constant rank \(n\), we have \(e(E)=c_n(E)\) (this statement follows from its nonequivariant version, see e.g.\ \cite[Example~3.3.2]{Ful-In-98}).
\end{para}

\begin{lemma}
\label{lemm:epsilon_1_Euler}
Let \(X\) be a smooth \(k\)-variety with an action of a diagonalizable group \(G\) of finite type over \(k\).
Assume that \(G\) contains \(\mu_p\) as a subgroup acting trivially on \(X\).
Let \(E\) be a \(G\)-equivariant vector bundle over \(X\) such that \(E^{\mu_p}=0\).
Then the Euler class \(e(E)\) is invertible in \(\Ch_{\mup}(X)[\Eu_{\mu_p}^{-1}]\).
Denoting by \(e(-E)\) its inverse, for any \(r \in \Fp \smallsetminus \{0\}\) we have
\[
\epsilon_r(e(-E)) \in \Fil_{G/\mu_p}(X).
\]
\end{lemma}
\begin{proof}
Since \(e(E' \oplus E'') = e(E') e(E'')\) for all \(\mu_p\)-equivariant vector bundles \(E',E''\) over \(X\), it follows from \rref{lemm:decomp_diag} that we may assume that there exists a \(G/\mup\)-equivariant vector bundle \(F \to X\) such that \(E \simeq F \otimes \character_c\) as \(\mu_p\)-equivariant vector bundles, for some fixed \(c \in (\Zz/p) \smallsetminus \{0\}\).
As the \(k\)-variety \(X\) decomposes \(G\)-equivariantly as a disjoint union of varieties over which the vector bundle \(E\) has constant rank, we may assume that \(F\) has constant rank \(n \in \Nn\).

Using the splitting principle, it is not difficult to see that we have in \(\Ch^n_{\mup}(X)\)
\[
e(E) = c_n(F \otimes \character_c) =  c_1(\character_c)^n + c_1(F) c_1(\character_c)^{n-1} + \dots +c_{n-1}(F) c_1(\character_c) + c_n(F).
\]
Since \(F\) carries the trivial \(\mup\)-action, its Chern classes \(c_1(F),\dots,c_n(F)\) are nilpotent in the ring \(\Ch_{\mup}(X)\) (being the images under the map \rref{p:Ch_trivial_alg} of elements of \(\Ch^i(X)\) with \(i>0\)).
Thus \(e(E)\) becomes invertible in \(\Ch_{\mup}(X)[\Eu_{\mu_p}^{-1}]\).
In addition the inverse of
\[
\epsilon_r(e(E)) =  r^n + c_1(F) r^{n-1} + \dots + c_n(F) \in \Ch(X)
\]
is a polynomial in the Chern classes of \(F\), hence belongs to \(\Fil_{G/\mu_p}(X)\).
\end{proof}

\section{Dividing partitions}
Recall that \(p\) is a prime number.
In this section we fix an integer \(q \in \Nn\smallsetminus \{0\}\).
\begin{definition}
For a partition \(\alpha=(\alpha_1,\dots,\alpha_m)\), let us define the integer
\begin{equation}
\label{eq:def_piq}
\pi_q(\alpha) = \Big\lfloor \frac{\alpha_1}{q} \Big\rfloor +\dots + \Big\lfloor \frac{\alpha_m}{q} \Big\rfloor.
\end{equation}
\end{definition}

\begin{para}
\label{p:pi_weight}
Note that for any partition \(\alpha\) we have 
\[
\pi_q(\alpha) \leq \Big\lfloor \frac{|\alpha|}{q} \Big\rfloor.
\]
\end{para}

\begin{para}
\label{p:piq_geq}
If \(\beta \geq \alpha\) (in the sense of \rref{p:ineq_part}), then \(\pi_q(\beta) \geq \pi_q(\alpha)\).
\end{para}

\begin{para}
\label{p:piq_sum}
For any partitions \(\alpha\) and \(\beta\) we have
\[
\pi_q(\alpha \cup \beta) = \pi_q(\alpha) + \pi_q(\beta).
\]
\end{para}

\begin{definition}
Let \(G\) be a linear algebraic group over \(k\) acting on a smooth \(k\)-variety \(X\).
For each \(m \in \Nn\), we define
\[
\Fil_G^{m,q}(X) \subset \Ch(X)
\]
as the subgroup generated by the elements of the form
\[
c_{\alpha^1}(E_1) \cdots c_{\alpha^s}(E_s)
\]
where \(E_1,\dots,E_s \in \im(K_0(X;G) \to K_0(X))\) with \(\pi_q(\alpha^1) +\cdots + \pi_q(\alpha^s) \geq m\).
\end{definition}

\begin{para}
Note that \(\Fil_G(X)=\Fil_G^{0,q}(X)\) (see \rref{def:Fil_G}), and that for every \(m \in \Nn\), we have 
\[
\Fil_G(X)\cdot \Fil_G^{m,q}(X) \subset \Fil_G^{m,q}(X).
\]
\end{para}

\begin{lemma}
\label{lemm:C_G_vb} Let \(G\) be a linear algebraic group over \(k\) acting on a smooth \(k\)-variety \(X\).
Let \(m \in \Nn\).
Then the subgroup \(\Fil_G^{m,q}(X) \subset \Ch(X)\) is generated by the elements of the form
\[
c_{\alpha^1}(E_1) \cdots c_{\alpha^s}(E_s),
\]
where \(\pi_q(\alpha^1) +\cdots + \pi_q(\alpha^s) \geq m\) and \(E_1,\dots,E_s\) are vector bundles over \(X\) admitting a \(G\)-equivariant structure.
\end{lemma}
\begin{proof}
Consider the ring filtration
\[
\cdots \subset S^{n+1} \subset S^n \subset\cdots \subset S^0 =\Fil_G(X)
\]
generated by the condition that \(c_\alpha(V) \in S^{\pi_q(\alpha)}\) for every partition \(\alpha\) and vector bundle \(V \to X\) admitting a \(G\)-equivariant structure.
Certainly \(S^n \subset \Fil_G^{n,q}(X)\) for every \(n \in \Nn\), and it will suffice to prove the reverse inclusion.

Every element of \(\im(K_0(X;G) \to K_0(X))\) is of the form \(E-F\), where \(E,F\) are vector bundles over \(X\) admitting a \(G\)-equivariant structure.
Fixing such \(E,F\), it will suffice to show that \(c_\alpha(E-F) \in S^{\pi_q(\alpha)}\) for every partition \(\alpha\).
Recall from \rref{p:g_inverse} (and \rref{ex:CF}) that for any partition \(\alpha \neq \varnothing\) we have
\[
c_\alpha(-F) = -\sum_{\substack{\beta \cup \gamma = \alpha \\ \beta \neq \alpha}} c_\beta(-F) c_\gamma(F).
\]
In view of \eqref{p:piq_sum} (and the multiplicativity of the filtration \((S^n)\)), we deduce by induction on the length of \(\alpha\) that \(c_\alpha(-F) \in S^{\pi_q(\alpha)}\) for any partition \(\alpha\) (the case \(\alpha = \varnothing\) being clear).
Next, for any partition \(\alpha\) we have the formula (see \rref{p:g_alpha_sum})
\[
c_\alpha(E-F) = \sum_{\beta \cup \gamma = \alpha} c_\beta(E) c_\gamma(-F),
\]
and again in view of \eqref{p:piq_sum}  (and the multiplicativity of the filtration \((S^n)\)) we deduce that \(c_\alpha(E-F) \in S^{\pi_q(\alpha)}\).
\end{proof}

\begin{para}
\label{p:phi}
Let us consider the polynomial
\begin{equation}
\label{eq:phi}
\phi(x) = x(x+c_1(\Lc_1)) \cdots (x+c_1(\Lc_{p-1})) \in \Ch_{\mup}(\Speck)[x],
\end{equation}
and the family of polynomials \(f=(f_i, i\in \Nn)\) given by
\begin{equation}
\label{eq:f_i}
f_i(x) =x^{i-p\lfloor \frac{i}{p} \rfloor} \phi(x)^{\lfloor \frac{i}{p} \rfloor}\in \Ch_{\mup}(\Speck)[x].
\end{equation}
If \(X\) is a smooth \(k\)-variety with trivial \(\mup\)-action and \(E \in K_0(X;\mup)\), by \rref{def:gen_CF} we thus have for every partition \(\alpha\) a class
\[
f_\alpha(E) \in \Ch_{\mup}(X).
\]
Observe that \(f_i(x)\) is mapped to \(x^i\) under the morphism \(\Ch_{\mup}(\Speck)[x] \to \Ch(\Speck)=\Fp[x]\) induced by the morphism \(\epsilon\) of \rref{p:epsilon}, and so for any partition \(\alpha\)  and any \(E \in K_0(X;\mup)\) we have
\begin{equation}
\label{eq:epsilon_f_c}
\epsilon(f_\alpha(E)) = c_\alpha(E) \in \Ch(X).
\end{equation}
\end{para}

\begin{para}
\label{p:deg_f_alpha}
Consider the grading of the polynomial ring \(\Ch_{\mup}(k)[x]\) induced by the grading of \(\Ch_{\mup}(\Speck)\) and by letting \(x\) have degree \(1\).
Since \(\phi(x)\) is homogeneous of degree \(p\), each \(f_i\) is homogeneous of degree \(i\).
Thus by \rref{p:g_graded} we have in the situation of \rref{p:phi}
\[
f_\alpha(E) \in \Ch^{|\alpha|}_{\mup}(X).
\]
\end{para}

\begin{lemma}
\label{lemm:f_alpha_E}
Let \(G\) be a diagonalizable group of finite type over \(k\) containing \(\mu_p\) as a subgroup.
Let \(X\) be a smooth \(k\)-variety with a \(G\)-action, and \(E \to X\) a \(G\)-equivariant vector bundle.
Assume that \(\mu_p\) acts trivially on \(X\).
Then for any partition \(\alpha\) and any \(r \in \Fp\), we have
\[
\epsilon_r(c_\alpha(E)) \in \Fil_{G/\mu_p}(X) \;\quad \text{and} \;\quad \epsilon_r(f_\alpha(E)) \in \Fil^{\pi_{pq}(\alpha),q}_{G/\mu_p}(X).
\]
\end{lemma}
\begin{proof}
We let \(g_i =x^i\), resp.\ \(g_i =f_i\) (see \rref{p:phi}), for \(i \in \Nn\).
By \rref{lemm:decomp_diag} we may find \(G/\mu_p\)-equivariant vector bundles \(F_c\) for \(c \in \Zz/p\) such that we have a \(\mu_p\)-equivariant decomposition
\[
E \simeq \bigoplus_{c \in \Zz/p} F_c \otimes \character_c.
\]
It follows from \rref{p:g_alpha_sum} that \(g_\alpha(E)\) is an \(\Fp\)-linear combination of elements of the form
\[
\prod_{c\in \Zz/p}g_{\alpha^c}(F_c \otimes \character_c),
\]
where \(\alpha^0,\dots,\alpha^{p-1}\) are partitions such that \(\alpha^0 \cup \dots \cup \alpha^{p-1} = \alpha\), and therefore by \rref{p:piq_sum}
\[
\pi_{pq}(\alpha^0) + \dots + \pi_{pq}(\alpha^{p-1}) = \pi_{pq}(\alpha).
\]
In view of the multiplicativity of the filtration \((\Fil^{m,q}_{G/\mu_p}(X))\), we may thus assume that there exists a \(G/\mu_p\)-equivariant vector bundle \(F\) such that \(E \simeq F \otimes \character_c\) for a fixed \(c\in \Zz/p\), as \(\mup\)-equivariant vector bundles.
Consider the family of polynomials \(g'=(g'_i, i \in \Nn)\), where
\[
g'_i(x) = g_i(x+c_1(\character_c)) \in \Ch_{\mup}(\Speck)[x].
\]
If \(L \to Y\) is a \(\mup\)-equivariant line bundle, with \(Y\) smooth, then
\[
g_i(c_1(L\otimes \character_c)) = g_i(c_1(L) +c_1(\character_c)) =g'_i(c_1(L)) 
\]
for all \(i \in \Nn\).
Thus it follows from the splitting principle that, for any partition \(\alpha\), we have
\[
g_\alpha(E) = g_\alpha(F \otimes \character_c)=g'_{\alpha}(F) \in \Ch_{\mu_p}(X).
\]
In case \(g_i=x^i\) for all \(i \in \Nn\), it follows from \rref{lemm:g:div} (with \(u_i=0\) for all \(i\)) that \(g_\alpha(E)=c_\alpha(E)\) is a \(\Ch_{\mup}(\Speck)\)-linear combination of the elements \(c_\beta(F)\), where \(\beta\) runs over the partitions.
It follows that \(\epsilon_r(c_\alpha(E))\) is an \(\Fp\)-linear combination of elements \(c_\beta(F)\) (by \rref{p:epsilon_retraction}, as \(F\) carries the trivial \(\mup\)-action), and thus belongs to \(\Fil_{G/\mu_p}(X)\).
This proves the first statement.

In case \(g_i=f_i\) for all \(i\in \Nn\), the formulas \eqref{eq:phi} and \eqref{eq:f_i} imply that each polynomial \(g_i'\) is divisible by \(x^{\lfloor i/p \rfloor}\).
Hence we deduce from \rref{lemm:g:div} that \(g_\alpha(E) = f_\alpha(E)\) is a \(\Ch_{\mup}(\Speck)\)-linear combination of the elements \(c_\beta(F)\), with \(\beta \geq \widetilde{\alpha}\), where \(\widetilde{\alpha}\) is the partition corresponding (upon removing zeroes) to the tuple \((\lfloor \alpha_1/p \rfloor,\dots,\lfloor \alpha_m/p \rfloor)\).
In view of \rref{p:epsilon_retraction}, it follows that \(\epsilon_r(f_\alpha(E))\) is an \(\Fp\)-linear combination of the elements \(c_\beta(F)\), for \(\beta \geq \widetilde{\alpha}\).
For such \(\beta\), we have \(\pi_q(\beta) \geq \pi_q(\widetilde{\alpha}) = \pi_{pq}(\alpha)\) (see \rref{p:piq_geq}), proving the second statement.
\end{proof}

\section{Fixed locus dimension}

\begin{para}(See \cite[Proposition~A.8.10~(1)]{CGP}.)
Let \(G\) be a linear algebraic group over \(k\), and let \(X\) be a \(k\)-variety with a \(G\)-action.
The \emph{fixed locus} is a closed subscheme $X^G \subset X$ such that for any $k$-variety $T$, the subset $X^G(T)\subset X(T)$ consists of those morphisms $T\to X$ which are $G$-equivariant with respect to the trivial $G$-action on $T$.
\end{para}

\begin{para}
If \(G\) is a diagonalizable group of finite type over \(k\), and \(X\) is a smooth \(k\)-variety with a \(G\)-action, then the \(k\)-variety \(X^G\) is smooth (see e.g.\ \cite[Proposition~A.8.10~(2)]{CGP}).
\end{para}

\begin{para}
Let \(X\) be a projective \(k\)-variety, with structural morphism \(f \colon X \to k\).
We will write
\[
\deg =f_* \colon \Ch(X) \to \Ch(k)=\Fp.
\]
If \(G\) acts on \(X\), we will also write
\[
\deg_G = f_* \colon \Ch_G(X) \to \Ch_G(k).
\]
Note that the forgetful morphism of \rref{p:epsilon} verifies
\begin{equation}
\label{eq:deg_epsilon}
\deg \circ \epsilon = \epsilon \circ \deg_G.
\end{equation}
If \(\mup\) acts trivially on \(X\), then for any \(r \in \Fp\) we have
\begin{equation}
\label{eq:deg_epsilon_r}
\deg \circ \epsilon_r = \epsilon_r \circ \deg_{\mup}.
\end{equation}
\end{para}

\begin{lemma}
\label{lemm:loc}
Let \(X\) be a smooth projective \(k\)-variety with an action of \(\mup\).
Denote by \(i \colon X^{\mup} \to X\) the inclusion of the fixed locus, and by \(N\) the normal bundle to \(i\).
Then the element \(e(N)\) is invertible in \(\Ch_{\mup}(X^\mup)[\Eu_{\mu_p}^{-1}]\) (see \rref{def:epsilon_r}); let us denote by \(e(-N)\) its inverse.
Assume that \(X\) has pure dimension \(n \in \Nn\),  and let \(y \in \Ch^s_{\mup}(X)\) with \(s \leq n\).
Then for any \(r \in \Fp \smallsetminus \{0\}\) we have
\[
\deg(\epsilon(y)) = \deg \big(\epsilon_r(e(-N)) \cdot \epsilon_r(i^*(y))\big) \in \Fp.
\]
\end{lemma}
\begin{proof}
By \cite[(1.4.9)]{conc} we have \(N^{\mup}=0\).
Thus by \rref{lemm:epsilon_1_Euler} the element \(e(N)\) is invertible in \(\Ch_{\mup}(X^\mup)[\Eu_{\mu_p}^{-1}]\).
We have \(\deg_{\mu_p}(y) \in \Ch^{s-n}_{\mu_p}(k)\) with \(s-n\leq 0\), and thus
\begin{equation}
\label{eq:epsilon_0_1:2}
\deg(\epsilon(y))   \overset{\eqref{eq:deg_epsilon}}{=} \epsilon(\deg_{\mu_p}(y))  \overset{\rref{p:epsilon_epsilon_0}}{=}  \epsilon_0(\deg_{\mu_p}(y)) \overset{\rref{p:epsilon_indep}}{=} \epsilon_r(\deg_{\mu_p}(y)).
\end{equation}
By the concentration theorem \cite[(4.5.7)]{conc} we have
\[
y = i_*(e(-N) \cdot i^*(y)) \in \Ch_{\mu_p}(X)[\Eu_{\mu_p}^{-1}].
\]
Since the following diagram commutes
\[
\xymatrix{
\Ch_{\mup}(X^\mup)[\Eu_{\mu_p}^{-1}] \ar[rr]^{i_*} \ar[rd]_{\deg_{\mup}} && \Ch_{\mup}(X)[\Eu_{\mu_p}^{-1}]  \ar[ld]^{\deg_{\mup}}\\
&\Ch_{\mup}(\Speck)[\Eu_{\mu_p}^{-1}]  &
}
\]
we deduce that
\[
\epsilon_r(\deg_{\mu_p}(y))  = \epsilon_r(\deg_{\mu_p}(e(-N) \cdot i^*(y))) \overset{\rref{eq:deg_epsilon_r}}{=} \deg(\epsilon_r(e(-N) \cdot i^*(y))).
\]
We conclude using \eqref{eq:epsilon_0_1:2} and the fact that \(\epsilon_r\) is a ring morphism.
\end{proof}

\begin{para}
\label{p:diag_pgroup}
We say that a diagonalizable group \(G\) over \(k\) is a \emph{finite diagonalizable \(p\)-group} if the cardinality of its character group \(\widehat{G}\) is a power of the prime \(p\).
\end{para}

\begin{lemma}
\label{cor:Fil_diag}
Let \(G\) be a finite diagonalizable \(p\)-group acting on a smooth projective \(k\)-variety \(X\).
Let \(q = |\widehat{G}|\).
Then for any \(m \in \Nn\), 
\[
\deg \Fil_G^{m,q}(X) \subset \deg\Big(\bigoplus_{i \geq m} \Ch^i(X^G)\Big).
\]
\end{lemma}
\begin{proof}
We proceed by induction on \(q\), the case \(q=1\) being clear, as then \(G=1\) and
\[
\Fil_G^{m,1}(X) \subset \bigoplus_{i \geq m} \Ch^i(X).
\]
So we assume that \(q >1\).
Then we may find an inclusion \(\mup \subset G\).
Consider partitions \(\alpha^1,\dots,\alpha^s\) such that
\begin{equation}
\label{eq:pi_geq_m}
\pi_{q}(\alpha^1) + \cdots + \pi_{q}(\alpha^s)  \geq m,
\end{equation}
and \(G\)-equivariant vector bundles \(E_1,\dots,E_s\) over \(X\).
Set 
\[
\gamma =  c_{\alpha^1}(E_1)\cdots c_{\alpha^s}(E_s) \in \Ch^n(X),
\]
where \(n = |\alpha_1| + \cdots + |\alpha_s|\).
To prove the lemma, in view of \rref{lemm:C_G_vb} it will suffice to show that 
\[
\deg \gamma \in \deg\Big(\bigoplus_{i \geq m} \Ch^i(X^G)\Big).
\]
Note that \(X\) decomposes \(G\)-equivariantly as \(X = X_0 \sqcup \cdots \sqcup X_s\), where each \(X_i\) has pure dimension \(i\).
Setting \(X_{s+1}= \cdots = X_n=\varnothing\), we may arrange that \(n \leq s\).
Then \(\deg \gamma = \deg (\gamma|_{X_n})\).
As \((X_n)^G\) is a closed subscheme of \(X^G\), we have \(\deg \Ch^i((X_n)^G)\subset \deg \Ch^i(X^G)\) for every \(i \in \Nn\).
Thus we may replace \(X\) with \(X_n\), and assume that \(X\) has pure dimension \(n\).

As \(q>1\), we may find a subgroup \(\mup \subset G\).
Set \(q'=q/p\), \(X'=X^\mup\) and \(G'=G/\mup\).
Consider the element
\[
\delta = f_{\alpha^1}(E_1)\cdots f_{\alpha^s}(E_s)\in \Ch_{\mu_p}(X).
\]
Then \(\deg \gamma = \deg ( \epsilon(\delta))\) by \eqref{eq:epsilon_f_c}.
As \(\delta \in \Ch^n_{\mu_p}(X)\) by \rref{p:deg_f_alpha}, we have by \rref{lemm:loc} (and using its notation)
\[
\deg ( \epsilon(\delta)) = \deg \big(\epsilon_1(e(-N)) \cdot \epsilon_1(f_{\alpha^1}(i^*E_1))\cdots \epsilon_1(f_{\alpha^s}(i^*E_s))\big).
\]
In view of \rref{lemm:f_alpha_E} and \rref{lemm:epsilon_1_Euler} we deduce that the product
\[
\epsilon_1(e(-N)) \cdot \epsilon_1(f_{\alpha^1}(i^*E_1)) \cdots \epsilon_1(f_{\alpha^s}(i^*E_s))
\]
belongs to (taking \rref{eq:pi_geq_m} into account)
\[
\Fil_{G'}(X') \cdot \Fil^{\pi_{pq'}(\alpha^1),q'}_{G'}(X') \cdots \Fil^{\pi_{pq'}(\alpha^s),q'}_{G'}(X') \subset \Fil^{m,q'}_{G'}(X').
\]
By induction we know that 
\[
\deg \Fil^{m,q'}_{G'}(X') \subset \deg\Big(\bigoplus_{i \geq m} \Ch^i(X'^{G'})\Big).
\]
Since \(X'^{G'} = X^G\), this concludes the proof.
\end{proof}

\begin{theorem}
\label{prop:floor_q}
Let \(G\) be a finite diagonalizable \(p\)-group (see \rref{p:diag_pgroup})
acting on a smooth projective \(k\)-variety \(X\).
Let \(q = |\widehat{G}|\).
Let \(\alpha^1,\dots,\alpha^s\) be partitions, and \(E_1,\dots,E_s \in \im (K_0(X;G) \to K_0(X))\) be such that
\[
\deg ( c_{\alpha^1}(E_1) \cdots c_{\alpha^s}(E_s)) \neq 0 \in \Fp.
\]
Then using the notation of \eqref{eq:def_piq}, we have
\[
\dim X^G \geq \pi_q(\alpha^1) + \dots + \pi_q(\alpha^s).
\]
\end{theorem}
\begin{proof}
Let \(m= \pi_q(\alpha^1) + \dots + \pi_q(\alpha^s)\).
The assumption of the theorem implies that \(\deg \Fil_G^{m,q}(X)\neq 0\), hence by \rref{cor:Fil_diag} we have \(\Ch^i(X^G) \neq 0\) for some \(i \geq m\), and in particular  \(\dim X^G \geq m\).
\end{proof}

\begin{definition}
Let \(X\) be a smooth projective \(k\)-variety, and \(\Tan_X\) its tangent bundle.
The \emph{modulo \(p\) Chern number} corresponding to a partition \(\alpha\) is
\[
c_\alpha(X) = \deg c_\alpha(-\Tan_X) \in \Fp.
\]
\end{definition}

\begin{corollary}
\label{cor:floor_q}
Let \(G\) be a finite diagonalizable \(p\)-group acting on a smooth projective \(k\)-variety \(X\).
Let \(q = |\widehat{G}|\).
Let \(\alpha=(\alpha_1,\dots,\alpha_m)\) be a partition such that \(c_\alpha(X) \neq 0 \in \Fp\).
Then
\[
\dim X^G \geq \Big\lfloor \frac{\alpha_1}{q} \Big\rfloor +\dots + \Big\lfloor \frac{\alpha_m}{q} \Big\rfloor.
\]
\end{corollary}

\begin{remark}
The conclusion of the corollary remains valid if we assume \(\deg c_\alpha(\Tan_X) \neq 0\) instead of \(\deg c_\alpha(-\Tan_X) \neq 0\).
\end{remark}

\begin{remark}
Note that, when \(\alpha\) is a partition,
\[
q\cdot \pi_q(\alpha) \geq |\alpha| - (q-1)\length(\alpha).
\]
In the situation of \rref{cor:floor_q}, if \(X\) is equidimensional, we must have \(|\alpha|=\dim X\), and thus 
\[
\dim X \leq q \cdot \dim X^G+(q-1)\length(\alpha).
\]
For this reason \rref{cor:floor_q} tends to be more useful for partitions of small length.
\end{remark}

\begin{remark}
Similar bounds were obtained in topology by Kosniowski--Stong for \(\Zz/2\)-actions \cite[p.\ 314]{KS}, and \((\Zz/2)^k\)-actions \cite[p.\ 737]{Kosniowski-Stong-Z2k}.
\end{remark}

\section{The cobordism ring}
\begin{definition}
When \(X\) is a smooth projective \(k\)-variety, we will write (using the notation of \rref{p:bb})
\[
\lc X \rc = \sum_{\alpha} c_\alpha(X) b_\alpha \in \Fp[\bb],
\]
where \(\alpha\) runs over the partitions.
The subset of those classes \(\lc X \rc\), where \(X\) runs over the smooth projective \(k\)-varieties, will be denoted by \(\Laz_p \subset \Fp[\bb]\).
This is a graded subring, where \(\deg b_i=-i\).
\end{definition}

\begin{remark}
	The ring \(\Laz_p\) can be identified with the coefficient ring of the universal commutative formal group law whose formal multiplication by \(p\) vanishes (see \cite[Proposition~3.2.5]{inv}).
For \(p=2\) this is Thom's unoriented cobordism \(MO^*\) (see \cite[Remark~1.2.2]{LM-Al-07}).
\end{remark}

\begin{para}
When \(\alpha\) is a partition, we will denote by 
\[
c_\alpha \colon \Fp[\bb] \to \Fp
\]
the map given by taking the \(b_\alpha\)-coefficient.
Thus, when \(X\) is a smooth projective \(k\)-variety, we have
\[
c_\alpha(\lc X \rc) = c_\alpha(X) \in \Fp.
\]
\end{para}

\begin{para}
\label{p:c_product}
If \(\alpha\) is a partition, we have for every \(x,y \in \Laz_p\)
\[
c_\alpha(xy) = \sum_{\beta \cup \gamma = \alpha} c_\beta(x) \cdot c_\gamma(y).
\]
\end{para}

\begin{definition}
\label{def:dim_q}
Let \(q \in \Nn \smallsetminus \{0\}\).
For an element \(x \in \Laz_p\), we set
\[
\dim_qx = \sup \{\pi_q(\alpha) | \text{ \(\alpha\) is a partition such that \(c_\alpha(x) \neq 0\)}\} \in \{-\infty\} \cup \Nn.
\]
\end{definition}

Theorem \rref{cor:floor_q} can be reformulated as follows:
\begin{theorem}
\label{th:deg_q}
Let \(G\) be a finite diagonalizable \(p\)-group acting on a smooth projective \(k\)-variety \(X\).
Let \(q = |\widehat{G}|\).
 Then \(\dim X^G \geq \dim_q \lc X\rc\).
\end{theorem}

\begin{remark}
	The bound of (\ref{th:deg_q}) is sharp.
Indeed it may be verified that every element \(x \in \Laz_p\) is represented by a smooth projective \(k\)-variety \(X\) with a \(G\)-action satisfying \(\dim_q x = \dim X^G\).
\end{remark}

In the remainder of this section, we give an interpretation of the function \(\dim_q\) in terms of cobordism generators.

\begin{para}
\label{p:Lp_basis}
Let \(N_p \subset \Nn\) be the set of integers \(i \geq 1\) such that \(i+1\) is not a power of \(p\).
It follows from \cite[Theorem~8.2]{Mer-Ori}
and \cite[II, \S7]{Adams-Stable}, that the \(\Fp\)-algebra \(\Laz_p\) is polynomial on variables indexed by \(N_p\).
Moreover, a family \(\ell_i \in \Laz_p^{-i}\) for \(i\in N_p\) constitutes a system of polynomial generators of \(\Laz_p\) if and only if \(c_{(i)}(\ell_i) \neq 0\) for each \(i\).
\end{para}

\begin{para}
\label{p:indec}
Recall that an element of \(\Laz_p\) is called \emph{indecomposable} if it does not belong to \(I^2\), where \(I \subset \Laz_p\) is the ideal generated by homogeneous elements of nonzero degrees.
An element \(x \in \Laz_p\) is indecomposable if and only if \(c_{(n)}(x) \neq 0 \in \Fp\) for some \(n \in \Nn\); if this is the case, then \(n+1\) is not a power of \(p\) (see e.g.\ \cite[(7.3.2)]{inv}).
\end{para}

\begin{para}
The polynomial algebra \(\Fp[X_i, i\in N_p]\) is graded by letting \(X_i\) be homogeneous of degree \(i\).
The degree of a polynomial will be denoted by \(\deg P\).
\end{para}

\begin{para}
\label{p:deg_q}
Let \(q \in \Nn \smallsetminus \{0\}\).
We define another grading of the polynomial algebra \(\Fp[X_i, i\in N_p]\) by letting the variable \(X_i\) be homogeneous of degree \(\lfloor i/q \rfloor\).
We will denote by \(\deg_q P\) the degree of a polynomial \(P\) with respect to this grading.
\end{para}

\begin{para}
\label{p:deg_q_pi_q}
For any nonempty partition \(\alpha=(\alpha_1,\dots,\alpha_m)\), the monomial \(X_\alpha = X_{\alpha_1} \cdots X_{\alpha_m}\) satisfies
\[
\deg X_\alpha = |\alpha| \quad \text{ and } \quad \deg_q X_\alpha = \pi_q(\alpha).
\]
\end{para}

\begin{para}
\label{p:deg_q_var}
If \(i \in \Nn \smallsetminus \{0\}\) we have \(\deg_q X_i \geq (i -(q-1))/q\).
\end{para}

\begin{para}
\label{p:def_succ}
Let \(\alpha,\beta\) be partitions.
Write \(\beta = (\beta_1,\ldots,\beta_s)\).
We say that \emph{\(\alpha\) refines \(\beta\)}, and write \(\alpha \succeq \beta\), when there exist partitions \(\alpha^1,\ldots,\alpha^s\) such that \(\beta_i = |\alpha^i|\) for all \(i=1,\ldots,s\) and \(\alpha = \alpha^1 \cup \cdots \cup \alpha^s\).
\end{para}

\begin{para}
\label{p:pi_dec}
It follows from \rref{p:pi_weight} and \rref{p:piq_sum} that \(\pi_q(\alpha) \leq \pi_q(\beta)\) whenever \(\alpha \succeq \beta\).
\end{para}

\begin{lemma}
\label{lemm:succ}
Let \(\ell_i \in \Laz_p^{-i}\) for \(i\in N_p\) be a family of polynomial generators of the \(\Fp\)-algebra \(\Laz_p\).
If \(\alpha,\beta\) are partitions such that \(c_\alpha(\ell_\beta) \neq 0\) then \(\alpha \succeq \beta\).
\end{lemma}
\begin{proof}
Write \(\beta = (\beta_1,\dots,\beta_s)\).
Then by \rref{p:c_product} we have
\[
c_\alpha(\ell_\beta) = c_\alpha(\ell_{\beta_1} \cdots \ell_{\beta_s}) = \sum_{\alpha^1 \cup \dots \cup \alpha^s=\alpha} c_{\alpha^1}(\ell_{\beta_1}) \cdots c_{\alpha^s}(\ell_{\beta_s}).
\]
Thus there exist partitions \(\alpha^1,\dots,\alpha^s\) such that \(\alpha^1 \cup \dots \cup \alpha^s=\alpha\) and \(c_{\alpha^i}(\ell_{\beta_i}) \neq 0\) for \(i=1,\dots,s\).
By degree reasons, we must have \(|\alpha^i| = \beta_i\) for \(i=1,\dots,s\), proving that \(\alpha \succeq \beta\).
\end{proof}

\begin{proposition}
\label{prop:dimq_degq}
Let \(\ell_i \in \Laz_p^{-i}\) for \(i\in N_p\) be a family of polynomial generators of the \(\Fp\)-algebra \(\Laz_p\).
Then for any \(P \in \Fp[X_i, i \in N_p]\), we have (see \rref{p:deg_q} and \rref{def:dim_q})
\[
\dim_q P(\ell_1,\dots) = \deg_qP.
\]
\end{proposition}
\begin{proof}
Let \(x = P(\ell_1,\dots)\).
Let us write
\[
P = \sum_{\beta} \lambda_\beta X_\beta, \quad \text{with \(\lambda_\beta \in \Fp\)},
\]
where \(\beta\) runs over the partitions.
If \(\alpha\) is a partition, we have by \rref{lemm:succ}
\begin{equation}
\label{eq:c_alpha_succ}
c_\alpha(x) = \sum_{\alpha \succeq \beta} \lambda_\beta c_\alpha(\ell_\beta).
\end{equation}
Let us consider the sets of partitions
\[
\mathcal{P} = \{ \beta \text{ such that } \lambda_\beta \neq 0\} \quad \text{ and } \quad \mathcal{X} = \{ \alpha \text{ such that } c_\alpha(x) \neq 0\}.
\]
If \(\alpha \in \mathcal{X}\), then it follows from \eqref{eq:c_alpha_succ} that there exists \(\beta \in \mathcal{P}\) such that \(\alpha \succeq \beta\).
Conversely, if \(\beta\) is a minimal element of \(\mathcal{P}\) with respect to the partial order \(\succeq\), then it follows from \eqref{eq:c_alpha_succ} that \(c_\beta(x) = \lambda_\beta\), and so \(\beta \in \mathcal{X}\).
We have proved that \(\mathcal{P}\) and \(\mathcal{X}\) have the same sets of minimal elements (with respect to the partial order \(\succeq\)).
Since by \rref{p:pi_dec} the function \(\pi_q\) is decreasing, we conclude that
\[
\dim_q x = \sup \{\pi_q(\alpha) | \alpha\in \mathcal{X}\} = \sup \{\pi_q(\beta) | \beta\in \mathcal{P}\} = \deg_q P.\qedhere
\]
\end{proof}

\begin{remark}
Proposition \rref{prop:dimq_degq} shows	that the function \(\dim_q\) arises as the degree with respect to the grading of \(\Laz_p\) given by letting each \(\ell_i\) have degree \(\lfloor i/q \rfloor\).
Note however that this grading depends on the choice of the family of generators (while the function \(\dim_q\) does not).
\end{remark}

\begin{remark}
Let us define another grading on \(\Fp[\bb]\), which we call \(q\)-grading, by letting \(b_i\) have degree \(\lfloor i/q \rfloor\); then the degree of an element \(x \in \Laz_p \subset \Fp[\bb]\) is \(\dim_q x\).
Observe that the subring \(\Laz_p \subset \Fp[\bb]\) is \emph{not} graded (with respect to the \(q\)-grading).
For instance, we have
\[
\lc \Pp^4 \rc = b_4 + b_2^2 +b_2b_1^2 \in \Laz_2 \subset \Fd[\bb].
\]
The component of degree \(1\) of \(\lc \Pp^4\rc \in \Fd[\bb]\), with respect to the \(2\)-grading, is \(b_2b_1^2\).
This element has degree \(-4\), with respect to the usual grading of \(\Fp[\bb]\).
However there exists no polynomial \(P \in \Fd[X_i|i\in N_2]=\Fd[X_2,X_4,X_5,\dots]\) satisfying \(\deg P =4\) and \(\deg_2P =1\).
This implies that \(b_2b_1^2\) does not belong to \(\Laz_2\).
\end{remark}

\section{Applications}
\label{sect:app}
In this section we draw concrete consequences of Theorem~\rref{th:deg_q}, by attempting to control the ratio \(n/d\), where \(n=\dim X\) and \(d=\dim X^G\).

\begin{para}
Let us fix an integer \(q \in \Nn \smallsetminus \{0\}\).
Let \(I \subset \Nn \smallsetminus \{0\}\) be a subset.
We define
\[
\rho_q(I) =
\begin{cases}
\inf_{i \in I} \frac{\lfloor i/q\rfloor}{i} \in \mathbb{R} & \text{if \(I \neq \varnothing\)}\\
1/q& \text{if \(I = \varnothing\)}.
\end{cases}
\]
\end{para}

\begin{remark}
It is not difficult to see that \(\rho_q(I) \in \mathbb{Q}\): indeed, observe that for a fixed \(r=0,\dots,q-1\), the function \(\Nn \smallsetminus \{0\}\to \mathbb{Q}\) given by \(a \mapsto a/(aq+r)\) is increasing, so that, letting \(i_r = \inf \{i\in I| i=r \mod q\} \in \Nn \cup \{\infty\}\), we have when \(I \neq \varnothing\),
\[
\rho_q(I) = \inf_{\substack{r \in \{0,\dots,q-1\}\\i_r \neq \infty}} \frac{\lfloor i_r/q\rfloor}{i_r}.
\]
\end{remark}

\begin{para}
\label{ex:rho}
Heuristically \(\rho_q(I)\) is an approximation of \(1/q\), which gets better when \(I\) contains no small integers not divisible by \(q\).
More precisely:
\begin{enumerate}[(i)]
\item \label{ex:rho:1}
We have \(\rho_q(I) \leq 1/q\).

\item If \(I \subset q\Nn\), then \(\rho_q(I)=1/q\).

\item \label{ex:rho:3}
Assume that \(I \neq \varnothing\), let \(i= \min I\), and write \(i=aq+r\) with \(a\in \Nn\) and \(r \in \{0,\dots,q-1\}\).
Then
\[
\frac{a}{aq+q-1}\leq \rho_q(I) \leq \frac{a}{aq+r}.
\]
In particular, if \(r=q-1\) (i.e.\ \(i=-1 \mod q\)), then \(\rho_q(I) = a/i\).

\item \label{ex:rho:4}
If \(I \cap \{1,\dots,q-1\} = \varnothing\), then \(\rho_q(I) \geq 1/(2q-1)\) (this follows from \rref{ex:rho:3}, where \(a \geq 1\)).

\item \label{ex:rho:5}
Assume that \(q=2\).
Then
\[
\rho_2(I) =
\begin{cases}
1/2 & \text{if \(I\) contains no odd integers}\\
\frac{\lfloor n/2 \rfloor}{n}& \text{if \(n\) is the smallest odd integer in \(I\)}.
\end{cases}
\]
Since the smallest odd integer in \(N_2\) is \(5\), it follows that
\[
\rho_2(I) \geq 2/5 \quad \text{ for any \(I \subset N_2\)}.
\]
Moreover, if \(I \subset N_2 \smallsetminus \{5\}\), then \(\rho_2(I) \geq 3/7\).
\end{enumerate}
\end{para}

\begin{para}
\label{p:pi_rho}
Let \(\alpha=(\alpha_1,\dots,\alpha_m)\) be a partition such that \(\alpha_1,\dots,\alpha_m \in I\).
Then
\[
\pi_q(\alpha) \geq \rho_q(I) \cdot |\alpha|.
\]
Thus, in view of \rref{p:deg_q_pi_q} we have
\[
\deg_q X_\alpha  \geq \rho_q(I) \cdot \deg X_\alpha.
\]
\end{para}

\begin{para}
\label{p:X_d_n}
For the rest of this section \(X\) will be a smooth projective \(k\)-variety of pure dimension \(n\), with an action of a finite diagonalizable \(p\)-group \(G\) over \(k\).
We let \(q=|\widehat{G}|\), and \(d=\dim X^G\).
\end{para}

\begin{proposition}
\label{prop:A_Np}
Let \(A \subset N_p\) and \(s\in \Nn\).
Let \(I(A) \subset \Laz_p\) be the ideal generated by the homogeneous elements of degrees \(-j\) for \(j\in A\).
If \(\lc X \rc \in \Laz_p\) does not belong to \(I(A)^{s+1}\), then
\[
d \geq \rho_q(N_p \smallsetminus A) \cdot (n - (q-1)s).
\]
\end{proposition}
\begin{proof}
Choose a family of polynomial generators \(\ell_i \in \Laz_p^{-i}\) for \(i\in N_p\) of the \(\Fp\)-algebra \(\Laz_p\), and write \(\lc X \rc = P(\ell_1,\dots)\) with \(P \in \Fp[X_i,i \in N_p]\).
Note that \(P\) is homogeneous of degree \(n\).
By assumption \(P\) contains as a monomial a nonzero multiple of \(X_{i_1} \cdots X_{i_t} X_\alpha\), where \(i_1,\dots,i_t \in A\) with \(t \leq s\), and \(\alpha=(\alpha_1,\dots,\alpha_m)\) is such that \(\alpha_1,\dots,\alpha_m \not \in A\).
Then
\begin{equation}
\label{eq:deg_P}
n = i_1 +\dots +i_t + \alpha_1 + \dots + \alpha_m.
\end{equation}
Let us write \(\rho = \rho_q(N_p \smallsetminus A)\).
Then
\begin{align*}
\deg_q P 
&\geq\deg_q X_{i_1} + \cdots + \deg_q X_{i_t} + \deg_q X_\alpha \\
&\geq \frac{i_1 -(q-1)}{q} +\dots +\frac{i_t - (q-1)}{q} + \deg_q X_\alpha&&\text{by \rref{p:deg_q_var}}\\
&\geq \frac{i_1 -(q-1)}{q} +\dots +\frac{i_t - (q-1)}{q} + \rho \cdot (\alpha_1+\dots+\alpha_m)&&\text{by \rref{p:pi_rho}}\\
&\geq \rho \cdot (i_1 + \dots + i_t - (q-1)t + \alpha_1+\dots+\alpha_m)&&\text{by \dref{ex:rho}{ex:rho:1}}\\
&\geq \rho \cdot (n - (q-1)s) &&\text{by \eqref{eq:deg_P}}.
\end{align*}
As \(d \geq \deg_q P\) by \rref{th:deg_q} and \rref{prop:dimq_degq}, the statement follows.
\end{proof}

\begin{example}
\label{ex:indec}
Take \(A=N_p\) and \(s=1\) in \rref{prop:A_Np}.
Then \(\rho_q(N_p \smallsetminus A)=\rho_q(\varnothing)=1/q\).
We obtain that if \(\lc X \rc \in \Laz_p\) is indecomposable (see \rref{p:indec}), then \(n < q(d +1)\).
\end{example}

\begin{example}
\label{ex:5Halves}
Take \(q=2\) and so \(G=\mu_2\).
Let \(A=\varnothing\) and \(s=0\) in \rref{prop:A_Np}.
Recall from \dref{ex:rho}{ex:rho:5} that \(\rho_2(N_2) = 2/5\).
Thus if \(\lc X \rc \neq 0 \in \Laz_2\), we obtain that \(n \leq 5d/2\).
This is the algebraic version of Boardman's Five-Halves Theorem \cite[Theorem 1]{Boardman-BAMS}, already proved in \cite[(8.1.6)]{ciequ}.
Note that in loc.\ cit.\ it was additionally assumed (and required for the proof) that the characteristic of \(k\) differs from \(2\).
\end{example}

Next, we discuss conditions on the cobordism class of \(X\) arising when the fixed locus has particularly low dimension.
\begin{corollary}
\label{cor:small_fixed}
Let \(L \subset \Laz_p\) be the \(\Fp\)-subalgebra generated by the homogeneous elements of degrees \(-1,\dots,2-q\).
For \(m \in \Nn\) denote by \(L_m \subset \Laz_p\) the ideal generated by \(L \cap \Laz_p^{-i}\) for \(i \geq m\).
If \(n \geq (2q-1)d\), then \(\lc X \rc \in L_{n-(2q-1)d}\).
\end{corollary}
\begin{proof}
Choose a family of polynomial generators \(\ell_i \in \Laz_p^{-i}\) for \(i\in N_p\) of the \(\Fp\)-algebra \(\Laz_p\), and write \(\lc X \rc = P(\ell_1,\dots)\) with \(P \in \Fp[X_i,i\in N_p]\).
Let \(M\) be a monomial of \(P\).
Note that, as \(q=|\widehat{G}|\) is a power of \(p\), we have \(q-1\not \in N_p\).
Thus \(M\) is a nonzero multiple of \(X_\alpha X_\beta\), where \(\alpha=(\alpha_1, \dots ,\alpha_m)\) and \(\beta = (\beta_1,\dots,\beta_s)\) are partitions such that \(\alpha_i \geq q\) for all \(i=1,\dots,m\), and \(q-2 \geq \beta_j\) for all \(j=1,\dots,s\).
Then by \dref{ex:rho}{ex:rho:4}, and as \(P\) is homogeneous of degree \(n\)
\[
\deg_q P \geq \deg_q (X_\alpha X_\beta) \geq \deg_q X_\alpha \geq \frac{\deg X_\alpha}{2q-1} = \frac{n - \deg X_\beta}{2q-1}.
\]
Since \(d \geq \deg_qP\) by \rref{th:deg_q} and \rref{prop:dimq_degq}, it follows that
\[
\deg X_\beta \geq n -(2q-1)d.
\]
This implies that \(\ell_{\beta} \in L_{n -(2q-1)d}\), so that \(\ell_{\alpha} \ell_{\beta} \in L_{n -(2q-1)d}\).
Therefore \(M(\ell_1,\dots) \in L_{n -(2q-1)d}\), which implies the statement.
\end{proof}

\begin{example}
Assume that \(q=3\), so that \(G=\mu_3\).
Then \(L\) is the \(\mathbb{F}_3\)-algebra generated by \(\lc \Pp^1 \rc\).
Thus if \(n \geq 5d\) then \(\lc X \rc\) is divisible by \(\lc \Pp^1 \rc^{n-5d}\) in \(\Laz_3\).
\end{example}

\begin{example}
Assume that \(q=4\), so that \(G \in \{\mu_4,\mu_2 \times \mu_2\}\).
Then \(L\) is the \(\mathbb{F}_2\)-algebra generated by \(\lc \Pp^2 \rc\), hence \(L_m\) is the ideal generated by \(\lc \Pp^2 \rc^{\lceil \frac{m}{2} \rceil}\).
Therefore if \(n \geq 7d\) then \(\lc X \rc\) is divisible by \(\lc \Pp^2 \rc^{\lceil \frac{n-7d}{2} \rceil}\) in \(\Laz_2\).
\end{example}

\begin{example}
\label{ex:isolated}
Consider the \(\Fp\)-subalgebra \(L_0 \subset \Laz_p\), generated by the homogeneous elements of degrees \(-1,\dots,2-q\).
Assume that \(G\) acts with isolated points on \(X\), in other words that \(d \leq 0\).
In this case \rref{cor:small_fixed} asserts that \(\lc X \rc \in L_0\).
\end{example}

Let us now turn to the case \(G=\mu_2\), and assume that \(\lc X \rc \neq 0 \in \Laz_2\).
Recall from \rref{ex:5Halves} that \(n \leq (5/2)d\).
As \(5/2 < 3=2q-1\), Corollary~\rref{cor:small_fixed} applies only when \(d=0\), and simply asserts that \(\lc X \rc \in \Fd \subset \Laz_2\) (which is the content of \rref{ex:isolated} for \(q=2\)).
We have the following replacement though:
\begin{corollary}
	Assume that \(G=\mu_2\).
Denote by \(H_{2,4} \subset \Pp^2 \times \Pp^4\) the Milnor hypersurface (see e.g.\ \cite[(5.1.3)]{ciequ}).
If \(n \geq 7d/3\) then \(\lc X \rc\) is divisible by \(\lc H_{2,4} \rc^{\lceil \frac{3n - 7d}{15} \rceil}\) in \(\Laz_2\).
\end{corollary}
\begin{proof}
	By e.g.\ \cite[(5.1.5)]{ciequ} we have \(c_{(5)}(\lc H_{2,4} \rc) = \binom{6}{2}=15\ne 0 \in \Fd\).
Therefore by \rref{p:Lp_basis}, there exists a family of polynomial generators \(\ell_i \in \Laz_2^{-i}\) for \(i\in N_2\) of the \(\Fd\)-algebra \(\Laz_2\) such that \(\ell_5 =\lc H_{2,4} \rc\).
Write \(\lc X \rc = P(\ell_2,\dots)\) with \(P \in \Fd[X_i, i\in N_2]\) homogeneous of degree \(n\).
Let \(M\) be a monomial of \(P\).
Then \(M\) is a nonzero multiple of \(X_\alpha X_5^s\), where \(s\in \Nn\) and \(\alpha=(\alpha_1, \dots ,\alpha_m)\) is a partition such that each \(\alpha_i\) belongs to \(N_2 \smallsetminus \{5\}\).
By \dref{ex:rho}{ex:rho:5} we have
\[
\deg_q P \geq \deg_q (X_\alpha X_5^s) \geq \deg_q X_\alpha \geq \frac{3 \deg X_\alpha}{7} = \frac{3(n - 5s)}{7}.
\]
Since \(d \geq \deg_qP\) by \rref{th:deg_q} and \rref{prop:dimq_degq}, we deduce that
\[
s \geq \frac{n-7d/3}{5} = \frac{3n -7d}{15}.
\]
Therefore \(M(\ell_2,\dots)\) is a multiple of \((\ell_5)^{\lceil \frac{3n - 7d}{15} \rceil}\), and the statement follows.
\end{proof}

\end{document}